\documentclass[11pt]{amsart}
\usepackage{amsmath,amsthm,amsfonts,amssymb, eucal, amscd}
\newtheorem{thm}[equation]{Theorem}

\newtheorem{lem}[equation]{Lemma}
\newtheorem{prop}[equation]{Proposition}
\newtheorem{pgraph}[equation]{} 
\theoremstyle{definition}
\newtheorem{defn}[equation]{Definition}

\newtheorem{rem}[equation]{Remark}

\newtheorem{asspt}[equation]{Assumption}
\newtheorem{subsec}[equation]{} 
\numberwithin{equation}{section}

\newcommand{\un}{\underline}
\newcommand{\K}{\mathbb K} 
\newcommand{\ot}{\otimes}
\newcommand{\one}  {\underline {1}}
\newcommand{\Ob}{\mathfrak{ob}}
\newcommand{\Mod}{\mathfrak{mod}}
\newcommand{\Dft}{\mathsf{D}(\un \phi,\un t)}
\newcommand{\Df}{\mathsf{D}}

\def \n {\noindent}

\def \m {\medskip} 
\newcommand{\Ars}{\mathsf{A}_{\un r, \un s}(n)}

\def \a {\alpha}

\def \be {\beta}

\def \e {\epsilon}

\def \gcd {{\operatorname {gcd}}}

\def \l {\lambda}

\def \lcm {{\operatorname {lcm}}}

\def \w {\omega}

\def \K {\Bbb K} 
\def \Z {\Bbb Z}

\usepackage[letterpaper,margin=1.48in]{geometry}
\begin{document}
\title{Multiparameter Weyl Algebras}
\date{}

\author[]{Georgia Benkart} \address{Department of Mathematics \\ University
of Wisconsin \\  Madison, WI  53706}
\email{benkart@math.wisc.edu}
\maketitle 
\begin{abstract}{We introduce a family of unital associative algebras  $\Ars$
which are multiparameter analogues of the Weyl algebras and 
determine the simple  weight modules and the Whittaker modules for $\Ars$.   All these modules can  be regarded as spaces of (Laurent) polynomials with certain $\Ars$-actions 
on them.

This paper was written in February 2008 and some copies of it were distributed, but it has never been posted or published. We thank Bryan Bischof and Jason Gaddis for their interest in the work and for  encouraging us to make the paper more widely available.   The references in this posted  version have been updated, and some cosmetic changes made;  in particular, some typos  have been corrected.  The investigations have been generalized by V. Futorny and J. Hartwig to  the context of multiparameter twisted Weyl algebras (see Journal of Algebra {\bf 357} (2012), 69-93).  
}  \end{abstract}

\section{Introduction}\label{sec:intro}

Let  $\un r
= (r_1,\dots,r_n)$ and $\un s = (s_1,\dots,s_n)$ be $n$-tuples of nonzero scalars
in a field $\K$.
We introduce a family of unital associative $\K$-algebras $\Ars$ which
are multiparameter analogues of the 
Weyl algebras.  Specializing all the  $r_i$ to equal $q$ and all the
$s_i$ to equal
$q^{-1}$ and factoring out by a certain ideal  gives Hayashi's $q$-analogues of the Weyl algebras (see \cite{H}). 
Our algebras have connections with the generalized Weyl algebras in (\cite{Bav1}, \cite{Bav2}, \cite{J})  and with
the down-up algebras in \cite{BenR},
and they have a natural action on the polynomial algebra in $n$ variables. 
We define Verma modules for the algebras  $\Ars$  and show
that the polynomial algebra is a Verma module.    We give a complete description of all the 
simple weight modules for 
$\Ars$ when $\K$ is algebraically closed 
and the only pair $(p,q) \in \mathbb Z^2$ such that $r_i^p = s_i^q$ is  $(0,0)$ for each $i$.      In the final
section, we determine the Whittaker modules for $\Ars$. 
All these modules can be regarded as spaces of (Laurent) polynomials with certain 
$\Ars$-actions on them.

\smallskip  
Assume $\un r = (r_1,\dots,r_n)$ and $\un s = (s_1,\dots,s_n)$ are $n$-tuples of nonzero scalars
in a field $\K$ such that  $(r_is_i^{-1})^2 \neq 1$ for each $i$.  
Let $\Ars$ be the 
unital associative algebra over the field $\K$ 
generated by elements
$\rho_i, \rho_i^{-1}, \sigma_i, \sigma_i^{-1}, x_i$, $y_i, \ i=1,\dots, n$, subject to
the following relations:  
\begin{itemize}
\item[(R1)] 
 The $\rho_i^{\pm 1}, \ \sigma_j^{\pm 1}$ all commute with one 
another and  $\rho_i\rho_i^{-1}=\sigma_i\sigma_i^{-1}=1;$  

\item[(R2)] 
$\rho_ix_j =r_i^{\delta_{i,j}}x_j \rho_i \qquad 
\rho_i y_j = r_i^{-\delta_{i,j}}y_j \rho_i \qquad    \quad 1 \leq i,j \leq n;$

\item[(R3)] $\sigma_ix_j =s_i^{\delta_{i,j}}x_j \sigma_i \qquad
\sigma_i y_j = s_i^{-\delta_{i,j}}y_j \sigma_i \qquad  
   \quad 1 \leq i,j \leq n;$

\item[(R4)]  $x_ix_j = x_jx_i, \qquad  y_iy_j = y_jy_i,  \qquad 1 \leq i,j \leq n; $

\noindent   $y_ix_j =  x_j y_i, \qquad  1 \leq i \neq j \leq n$;

\item[(R5)]   
$y_ix_i - r_i^2 x_iy_i = \sigma_i^2$ \  and \
$y_ix_i - s_i^2 x_iy_i = \rho_i^2$,   \qquad $1 \leq i \leq n$,

\n or equivalently
\smallskip

\item[(R5')]   
$\displaystyle{y_ix_i =  \frac{ r_i^2\rho_i^2 -s_i^2 \sigma_i^2}{r_i^2 -s_i^2}}$ \  and \
$\displaystyle{x_iy_i =  \frac{\rho_i^2 -\sigma_i^2}{r_i^2 -s_i^2} \qquad 1 \leq i \leq n.}$
\end{itemize}  

When $r_i = q$ and $s_i = q^{-1}$ for all $i$, we may  
 quotient by the ideal generated by
the elements  $\rho_i \sigma_i -1$, $i=1,\dots,n$,  to obtain Hayashi's algebra $\mathsf{A}_q^{-}(n)$.
(The generators $\rho_i, \sigma_i,  x_i, y_i$  are identified with Hayashi's  $\omega_i, \omega_i^{-1}, \psi_i^\dagger, \psi_i$,  respectively.)      
\smallskip

The elements $\rho_i$, $\sigma_i$ act as automorphisms
on $\Ars$ by conjugation.  These automorphisms fix
the elements $\rho_j$, $\sigma_j$ for
all $j$ and the elements $x_j,y_j$ for $j \neq i$.  Moreover,
\begin{eqnarray}  \rho_i x_i \rho_i^{-1} = r_i x_i  \qquad \qquad 
\rho_i y_i \rho_i^{-1} &=& r_i^{-1} y_i,  \label{eq:conj1} \\
\sigma_i x_i \sigma_i^{-1} = s_i x_i\qquad \qquad 
\sigma_i y_i \sigma_i^{-1} &=& s_i^{-1}y_i.  \label{eq:conj2}   \end{eqnarray}

\n Thus  $\rho_i^p \sigma_i^{-q}$ is the identity automorphism for some
$p,q
\in
\mathbb Z$ if and only if $r_i^p = s_i^q$.   
\medskip
 
 \begin{asspt}\label{asspt1}  Henceforth we assume that 
for  each $i =1,\dots,n$,  the only pair $(p,q) \in \mathbb Z^2$ such that 
$r_i^p = s_i^q$ is the pair $(0,0)$.   \end{asspt}
  
Under this assumption, we may identify
the elements $\rho_i, \sigma_i$ with the automorphisms they determine.

\section{Connections with generalized Weyl algebras} \label{sec:GWA} 

A {\it generalized Weyl algebra} $\Dft$ {\it of degree $n$}  is constructed
from a commutative algebra
$\mathsf{D}$ over $\K$, an $n$-tuple $\un \phi = (\phi_1,\dots, \phi_n)$ of
commuting  automorphisms of $\mathsf{D}$, and an $n$-tuple
$\un t = (t_1,\dots,t_n)$ of nonzero elements $t_i  \in \Df$.  Then $\Dft$
is the unital associative algebra
generated over $\mathsf{D}$ by $2n$ elements, $X_i,Y_i$, $i=1,\dots,n$,  subject to the relations  
\smallskip
\begin{itemize}
\item[(W1)]  $Y_i X_i = t_i,  \qquad   \qquad  \ X_i Y_i = \phi_i(t_i),$
\item[(W2)]  $X_i d = \phi_i(d) X_i,  \qquad Y_i d = \phi_i^{-1}(d)Y_i,$  for all $d \in \mathsf{D}$, 
\item[(W3)]  $X_iX_j = X_jX_i,   \qquad  \  Y_iY_j = Y_jY _i,  \qquad 
 \ (1 \leq i,j \leq n),$  
 \item[(W4)] $X_iY_j =  Y_j X_i,  \qquad  \ \  (1 \leq i \neq j \leq n).$  \end{itemize} 
\smallskip

When $\mathsf{D} =\Df_1 \times \Df_2 \times \cdots \times \Df_n$, $t_i \in \Df_i$, and $\phi_i$
is an automorphism of $\Df_i$ extended to $\Df$ by
having $\phi_i$ act as the identity automorphism on $\Df_j$ for $j \neq i$,  the algebra 
  $\Dft$ is isomorphic to the tensor product (over $\K$)
of $n$ degree one generalized Weyl algebras, 
$\Df_1(\phi_1,t_1) \ot \cdots \ot  \Df_n(\phi_n,t_n)$, where $\Df_i(\phi_i,t_i)$ has
generators $X_i, Y_i$. 
\smallskip 

The multiparameter Weyl algebra $\Ars$ can be realized as 
a degree $n$ generalized Weyl algebra.   For this construction, let
$\Df_i$ be the subalgebra of $\Ars$ generated by the 
elements $\rho_i,\rho_i^{-1},\sigma_i,
\sigma_i^{-1}$.  Thus, $\Df_i$ is isomorphic to $\K[\rho_i^{\pm 1},\sigma_i^{\pm 1}]$.   Set $\Df = \Df_1 \times \Df_2 \times \cdots \times \Df_n$. 
Let  $\phi_i$ be the automorphism of $\Df$ given by

\begin{equation}\label{eq:phiidef} \phi_i(\rho_j) = r_i^{-\delta_{i,j}}\rho_j \qquad \qquad \phi_i(\sigma_i) = s_i^{-\delta_{i,j}}\sigma_i. \end{equation}
 
Now set
\begin{equation}\label{eq:tidef}t_i = \frac{r_i^2\rho_i^2 -s_i^2\sigma_i^2}{r_i^2 -s_i^2}, \qquad
X_i = x_i, \qquad Y_i = y_i, \end{equation}  
and observe that 

$$Y_iX_i = t_i,  \qquad \text{and} \qquad  X_iY_i = \frac{ \rho_i^2 - \sigma_i^2}{r_i^2 -s_i^2} =
\phi_i(t_i)$$ 
\noindent  are just the relations in (R5)'.  The relations in (R1) and (R4) are apparent. 
The identities in (R2) and (R3) are equivalent to the
statements  $Y_j d = \phi_j^{-1}(d)Y_j, \ \ X_jd = \phi_j(d)X_j$ with
$d = \rho_i$ and $\sigma_i$.      Therefore, there is a surjection $\mathsf{W}_n: = \Dft
\rightarrow \Ars$. But since $\Ars$ has a presentation
by (R1)-(R5), there is a surjection $\Ars \rightarrow
\mathsf{W}_n$.   Since that map is the inverse of the other one, these
algebras are isomorphic.  

 Bavula \cite[Prop.~7]{Bav1}  has shown  that  a generalized
Weyl algebra  $\Dft$  is left and right
Noetherian if $\Df$ is Noetherian, and it is a domain if $\Df$ is a domain. Since $\Df$ is commutative
and finitely generated, it is Noetherian (see \cite[Prop.~1.2]{GW}),
hence so are $\mathsf{W}_n$ and $\Ars$. 
By Assumption \ref{asspt1}, $\Df$ is a domain since it can be identified with the Laurent polynomial algebra
$\K[\rho_i^{\pm 1}, \sigma_i^{\pm 1} \mid i =  1,\dots, n]$;  hence $\Ars$ is
a domain also.      In summary, we have 
\begin{prop}\label{prop:GWA} When the parameters $r_i,s_i$ satisfy  Assumption \ref{asspt1},  the  multiparameter Weyl algebra $\Ars$ is isomorphic
to the degree $n$ generalized Weyl algebra $\mathsf{W}_n = \Dft$,  where
$\Df$ is the $\K$-algebra generated by the elements $\rho_i,\rho_i^{-1},\sigma_i,
\sigma_i^{-1}$, $i=1,\dots,n$, subject to the relations in \hbox{\rm (R1)},  $\phi_i$ is as in \eqref{eq:phiidef} for each i;
and the elements $t_i$ are as in \eqref{eq:tidef}. Thus, $\Ars$
is a left and right Noetherian  domain.      \end{prop}
  
\begin{thm}\label{thm:Asimple}   Under Assumption \ref{asspt1},  $\Ars$ is a simple algebra.  \end{thm}

\begin{proof}  We will invoke a result of Jordan \cite[Thm.~6.1]{J} which provides
a criterion for the simplicity of a degree one generalized Weyl algebra  $\mathcal D(\varphi,\tau)$. 
Such an algebra $\mathcal D(\varphi,\tau)$ is simple if (i) $\mathcal D$ is
a commutative Noetherian ring; (ii) $\varphi$ has infinite order; 
(iii) $\tau$ is regular; (iv) $\mathcal D$ has no $\varphi$-invariant ideals except 0 and $\mathcal D$;
and (v) for all positive integers $m$, $\tau \mathcal D + \varphi^m(\tau) \mathcal D = \mathcal D$. 

  In applying this result, we will take $\mathcal D$ to be one of the algebras $\Df_i$ 
and will omit the subscript $i$.  Thus,  we will  
suppose $\Df$ is generated by $\rho^{\pm 1}$ and
$\sigma^{\pm 1}$. Let $\phi$ be as in \eqref{eq:phiidef}, and $t$ be as in \eqref{eq:tidef}.
The assumptions of Theorem \ref{thm:Asimple} imply that $\Df = \K[\rho^{\pm 1},\sigma^{\pm 1}]$  is a commutative, Noetherian
domain, which gives (i) and (iii).   Moreover, since  $r$ and $s$
are not  roots of unity, we have (ii).  Assume $\mathsf{J}$ is a $\phi$-invariant
ideal of $\Df$.  Let $v = \sum_{k,\ell} c_{k,\ell} \rho^k \sigma^\ell$ be a nonzero element of $\mathsf{J}$ 
with a minimal number of nonzero summands.  Clearly, if there is only one such
summand, then $1 \in \mathsf{J}$ and $\mathsf{J}= \Df$, so we may assume there are at least two
nonzero summands.  Let $(k',\ell')$ be a pair such that
$c_{k',\ell'} \neq 0$.    Applying $\phi$ we have
$\phi(v) = \sum_{k,\ell} c_{k,\ell} r^{-k}s^{-\ell}\rho^k \sigma^\ell$. But then
$r^{k'} s^{\ell'} \phi(v) - v$ has fewer summands.  It is nonzero since
$r^{k'-k}s^{\ell'-\ell} \neq 1$ whenever $(k,\ell) \neq (k',\ell')$.  This shows that
a minimal sum must have only one term and $\mathsf{J} = \Df$. Therefore, $\Df$ has no non-trivial
$\phi$-invariant ideals.   Finally, for (v) note that
\begin{eqnarray*}
&&\hspace{-.8 truein}  t(r^2-s^2)r^{-2m}\sigma^{-2} - \phi^m(t)(r^2-s^2)\sigma^{-2} \\
& =& (r^2\rho^2 -s^2 \sigma^2)r^{-2m}\sigma^{-2} - (r^{2-2m}\rho^2 -
s^{2-2m}\sigma^2)\sigma^{-2} \\
& =&  s^2(s^{-2m}-r^{-2m}) \neq 0, \end{eqnarray*}  
\noindent  so that $t\Df+ \phi^m(t)\Df = \Df$ for all $m$.   Therefore, by Jordan's theorem, 
$\Df_i[\phi_i,t_i]$ is simple for each $i$.   As $\Ars \cong
\Df_1[\phi_1,t_1] \ot \cdots \ot \Df_n[\phi_n,t_n]$, the result follows. 
\end{proof}

\smallskip
The algebra $\Ars$ has a natural action on 
the $n$ variable polynomial
algebra. Consider $n$-tuples $\un k = (k_1,\dots,k_n)$ of nonnegative integers.
Let $\e_i$ be the $n$-tuple with $1$ in the $i$th position and 0
as the rest of the components.  Let $\mathsf{P}(n) = \mathbb K[z_1,\dots,z_n]$,  the space
of polynomials with the basis 
consisting of the monomials, 
$$z(\un k) = z_1^{k_1} z_2^{k_2} \cdots z_n^{k_n},  \quad \un k \in  \mathbb N^n,$$

\n and with the $\Ars$-action given by
\begin{eqnarray}\label{eq:acts}
\rho_i z(\un k) &=& r_i^{k_i}z(\un k), \\
 \sigma_i z(\un k) &=& s_i^{k_i}z(\un k), \nonumber \\
x_i z(\un k) &=&  z(\un k+\e_i), \nonumber  \\
y_i z(\un k) &=& [k_i] z(\un k-\e_i), \nonumber 
\nonumber  \end{eqnarray}
\n where  
\begin{equation}[k] = \frac {r_i^{2k} - s_i^{2k}}{r_i^2 - s_i^2}.\end{equation}
\n Under Assumption \ref{asspt1},  $\mathsf{P}(n)$ is a simple module for $\Ars$.
Indeed, the vectors $z(\un k)$ are common eigenvectors for the $\rho_i$ and $\sigma_i$,
and each  $z(\un k)$ determines a one-dimensional eigenspace.   Any nonzero submodule must contain one of the
vectors $z(\un k)$, and then applying \eqref{eq:acts}, we see it must contain
all such basis vectors.  \m

  \section{ Connections with down-up algebras}
 
Down-up algebras were introduced in (\cite{BenR}, \cite{Ben})  as generalizations of the 
algebra generated by the down and up operators on a partially ordered set.
They are unital associative $\mathbb K$-algebras $\mathsf{A}(\alpha,\beta,\gamma)$ with generators $d,u$ which
satisfy the relations
\begin{eqnarray*}
d^2 u &=& \alpha dud + \beta ud^2 + \gamma d \\
du^2 &=& \alpha udu + \beta u^2 d + \gamma u, \end{eqnarray*} 
\n where $\alpha,\beta,\gamma$ are fixed scalars from $\mathbb K$.  These algebras exhibit
many beautiful properties. For example, they have a Poincar\'e-Birkhoff-Witt type basis,
$$\{u^i (du)^j d^k \mid i,j,k \in \mathbb N \},$$
\n and Gelfand-Kirillov dimension 3 (see \cite{BenR}  and \cite{Ben}).   
They are left and right Noetherian domains if and only if
$\beta \neq 0$.  In that case, they can be realized as generalized
Weyl algebras of degree one \cite{KMP}  or as certain hyperbolic rings \cite{Ku}. 
Also when $\beta \neq 0$, they have Krull dimension 2 if and only
if char($\mathbb K$) = 0, $\gamma \neq 0$ and $\a+\be = 1$.  Otherwise the Krull
dimension is 3 (see \cite{BL}).  
\smallskip

 In the algebra 
$\Ars$,  
multiplying the equations in (R5)' by $y_i$ and $x_i$ on the left and
right and using (R2) shows that the following relations hold for each $i$:

\begin{eqnarray}
y_i^2 x_i & =&  (r_i^2+s_i^2)y_ix_iy_i - r_i^2s_i^2 x_i y_i^2  \\
y_ix_i^2 & =&  (r_i^2+s_i^2)x_iy_ix_i - r_i^2s_i^2 y_i x_i^2.  \nonumber \end{eqnarray}
Thus, the elements $x_i$ and $y_i$ satisfy the defining relations 
of the down-up algebra $\mathsf{A}(r_i^2+s_i^2, -r_i^2s_i^2, 0)$.  
\begin{section}{Connections with multiparameter quantum groups}\end{section}

 We assume here that
$(\,,\,)$ is a $\mathbb Z$-bilinear form on $\mathbb Z^n$ relative to which 
the $\e_i$'s are orthonormal.  We adopt the
notation that $\alpha_i = \e_i-\e_{i+1}$ for $1 \leq i < n$.  We introduce a family of algebras
which generalize Takeuchi's two-parameter quantum groups (see \cite{T} and \cite{BW}).
\smallskip
 
As before let $\un r$ and $\un s$ be $n$-tuples of nonzero scalars.  The 
unital associative $\mathbb K$-algebra 
$\mathsf{U}_{\un r, \un s}(\mathfrak{sl}_n)$ is the algebra
generated by  $e_i,f_i, \omega_i^{\pm 1}
(\omega_i')^{\pm 1}, \ 1 \leq i < n$, subject to the following relations: 
\smallskip

\begin{itemize}
\item[{(U1)}]
 The $\w_i^{\pm 1}, \ (\w_j')^{\pm 1}$ all commute with one 
another and

\noindent $\w_i\w_i^{-1}=\w'_i(\w'_i)^{-1}=1;$
 
\item[{(U2)}] $\w_ie_j = r_i^{(\epsilon_i,\a_j)}s_{i+1}^{(\e_{i+1},\a_j)}e_j \w_i, $ 
 
\noindent $\w_i f_j =r_i^{-(\epsilon_i,\a_j)}s_{i+1}^{-(\e_{i+1},\a_j)}f_j \w_i;$

\item[{(U3)}]
$\w_i'e_j = r_{i+1}^{(\epsilon_{i+1},\a_j)}s_{i}^{(\e_{i},\a_j)}e_j \w_i' ,$

\noindent  $\w_i'f_j = r_{i+1}^{-(\epsilon_{i+1},\a_j)}s_{i}^{-(\e_{i},\a_j)}e_j \w_i'$; 
\item[{(U4)}]
$[e_i,f_j] =  
\delta_{i,j} \displaystyle{\frac{\omega_i^2-(\omega_i')^2} {r_i^2-s_i^2}};$

\item[{(U5)}] $[e_i,[e_i,e_{i+1}]_{r_{i+1}^2}]_{s_{i+1}^2} = 0 \qquad
[[e_i,e_{i+1}]_{r_{i+1}^2},e_{i+1}]_{s_{i+1}^2} = 0;$

\item[{(U6)}]
$[f_i,[f_i,f_{i+1}]_{r_{i+1}^{-2}}]_{s_{i+1}^{-2}} = 0 \qquad
[[f_i,f_{i+1}]_{r_{i+1}^{-2}},f_{i+1}]_{s_{i+1}^{-2}} = 0,$ 
\end{itemize}

\n for all $1 \leq i,j < n$,  with the convention that $[x,y]_q = xy-qyx$.  
 
\begin{rem}  Takeuchi's algebras are just the
case that $r_i =r$ and $s_i = s$ for all $i$.  When
$r_i= q$ and $s_i = q^{-1}$ for all $i$,  then after factoring out the ideal
generated by the elements $\omega_i'-\omega_i^{-1}$, 
the resulting algebra is isomorphic to $\mathsf{U}_q(\mathfrak{sl}_n)$.  In \cite{T} and \cite{BW}, relation (U4) reads
$$[e_i,f_j] =  
\delta_{i,j} \displaystyle{\frac {\omega_i-\omega_i',} {r-s}},$$
 \n and the subscripts in (U5) and (U6) are $r$ and $s$ rather than what is above.  This
difference is minor - it amounts to replacing $r,s$ in those papers by $r^2,s^2$ and then
choosing the bilinear form so that $(\e_i,\e_j) = \delta_{i,j}/2$. 
\end{rem}

There exists an algebra homomorphism
 $\mathsf{U}_{\un r,\un s}(\mathfrak{sl}_n) \rightarrow \Ars$  given by
\begin{eqnarray}\label{eq:maps}
&& \w_i \mapsto \rho_i \sigma_{i+1}, \\
&& \w_i' \mapsto \rho_{i+1}\sigma_i.  \nonumber \\
&& e_i \mapsto y_{i+1}x_i, \nonumber \\
&& f_i \mapsto y_i x_{i+1}, \nonumber 
 \end{eqnarray}

\n Verifying this is straightforward.  We present the argument only for (U5):
\begin{eqnarray*} 
&& \hspace{-.3 truein} [e_i,[e_i, e_{i+1}]_{r_{i+1}^2}]_{s_{i+1}^2} \mapsto  \\
&& \hspace{-.2 truein} 
y_{i+2}x_i^2\Big(y_{i+1}^2x_{i+1} - (r_{i+1}^2+s_{i+1}^2)y_{i+1}x_{i+1}y_{i+1}
+r_{i+1}^2s_{i+1}^2x_{i+1}y_{i+1}^2
\Big) = 0.\end{eqnarray*}
The rest of the calculations for (U5) and (U6) are virtually identical to these and just involve
the down-up relations or the variations on them obtained by multiplying through by
$r_{i+1}^{-2}s_{i+1}^{-2}$.   
\m

\begin{subsec}{\bf Weight modules for $\mathsf{U}_{\un r,\un s}(\mathfrak{sl}_n)$}\end{subsec}

Let $\mathsf{M}$ be a module for $\mathsf{U}_{\un r,\un s}(\mathfrak{sl}_n)$.
A {\it weight} $(\eta,\vartheta)$ of $\mathsf{M}$ consists of  two tuples $\eta = (\eta_1, \dots, \eta_n)$
and $\vartheta=(\vartheta_1,\dots,\vartheta_n)$ in  $ \K^{n-1}$  such that
$$\mathsf{M}_{(\eta,\vartheta)}=\{v \in \mathsf{M} \mid \omega_i v =\eta_i v, \  
\omega_i' v = \vartheta_i v \ \text{for all} \ i \} \neq 0.$$
\n $\mathsf{M}$ is a {\it weight module} if $\mathsf{M} = \bigoplus_{(\eta, \vartheta)} \mathsf{M}_{(\eta,\vartheta)}$.
\smallskip

When we regard the polynomial algebra $\mathsf{P}(n)$ as a module for $\mathsf{U}_{\un r,\un s}(\mathfrak{sl}_n)$ 
using the homomorphism in \eqref{eq:maps}, we obtain:
 \begin{eqnarray}
&& \omega_i  z(\un k) = r_i^{k_i}s_{i+1}^{k_{i+1}} z(\un k), \\
&& \omega_i' z(\un k) = r_{i+1}^{k_{i+1}}s_i^{k_i} z(\un k). \nonumber \\
&& e_i z(\un k) = [k_{i+1}]z(\un k+\e_{i}-\e_{i+1}), \nonumber \\ 
&& f_i z(\un k) = [k_i] z(\un k +
\e_{i+1} - \e_i),  \nonumber  
 \end{eqnarray}

\noindent  From this we can see that $\mathsf{P}(n)$ is a weight module,  as each monomial
$z(\un k)$ is a weight vector.  The weights are all distinct by Assumption \ref{asspt1}.  \smallskip

The monomials $z(\un k)$ such that $\sum_i k_i = m$ form a $\mathsf{U}_{\un r,\un s}(\mathfrak{sl}_n)$-submodule
$\mathsf{P}(n)_m$ and $\mathsf{P}(n) = \bigoplus_{m=0}^\infty \mathsf{P}(n)_m$.   The module $\mathsf{P}(n)_m$ is simple for each $m$.  To see this, note that any nonzero $\mathsf{U}_{\un r,\un s}(\mathfrak{sl}_n)$-submodule $\mathsf{Q}$ 
of $\mathsf{P}(n)_m$ will decompose into weight spaces.  Hence $\mathsf{Q}$ must contain
some $z(\un k)$.  Since $[\ell] \neq 0$ for $\ell > 0$, we can apply the
operators $e_i$ and $f_i$ to show that any monomial of degree $m$ belongs $\mathsf{Q}$.  
The vector $z(m\e_1)$ is annihilated by all the $e_i$. Thus, since it
is a weight vector and generates $\mathsf{P}(n)_m$, it is a highest weight vector for that module.  
\smallskip
 
\begin{section}{Verma modules}\end{section}
 
Assume $\lambda = (\lambda_1, \dots, \lambda_n)$ is a fixed $n$-tuple of elements of $\mathbb K$.
Let $\Omega$ be the group of automorphisms of $\Ars$ generated by
the elements $\rho_i^{\pm 1}, \sigma_i^{\pm 1}, \ i =1,\dots,n$. 
Suppose $\zeta: \Omega \rightarrow \{\pm 1\}$ is a group homomorphism.  
Let $\mathsf{V}(\lambda, \zeta)$ be the
$\K$-vector space having basis $\{v(\un k) \mid
\un k \in \mathbb N^n\}$ and having the following 
$\Ars$-module action:  
\begin{eqnarray}\label{eq:vermav}
\rho_i v(\un k) & =&  r_i^{k_i}\l_i \zeta_i v(\un k),     \\
\sigma_i v(\un k) & =&  s_i^{k_i}\l_i\zeta_i' v(\un k), \nonumber \\
x_i v(\un k) & =& v(\un k + \e_i),  \nonumber  \\
y_i v(\un k) & =& [k_i]\l_i^2 v(\un k - \e_i),  \nonumber  
\end{eqnarray} 

\n \n where $\zeta_i = \zeta(\rho_i)$
and $\zeta_i' = \zeta(\sigma_i)$.  Our convention is that $v(\un \ell) = 0$ whenever $\un \ell \not \in
\mathbb N^n$. 

One special case is where $\lambda = \one$, the tuple of all 1's, and
$\zeta(\rho_i) = 1 = \zeta(\sigma_i)$ for all $i$.  In this case, $\mathsf{V}(\lambda,\zeta) = \mathsf{P}(n)$. 

To show that $\mathsf{V}(\lambda,\zeta)$ is  a module for $\Ars$, we need to verify that the relations
(R1)-(R4) and (R5)' hold.   These routine calculations are omitted.  
\bigskip

\begin{section}{  Weight modules for $\Ars$} \end{section}
 
In this section we use the realization of $\Ars$ as a generalized
Weyl algebra $\mathsf{A} = \Df(\un \phi, \un t)$ to determine the simple weight modules of $\Ars$.
Thus, $\mathsf{D} = \mathbb K[\rho_i^{\pm 1},\sigma_i^{\pm 1}\mid i=1,\dots,n]$, 
$t_i =   \displaystyle{\frac{r_i^2 \rho_i^2 - s_i^2 \sigma_i^2}{r_i^2 - s_i^2}}$, 
and $\phi_i$ is the automorphism of $\mathsf{D}$ given by 
$\phi_i(\rho_j) = r_i^{-\delta_{i,j}} \rho_j$ and $\phi_i(\sigma_j) = s_i^{-\delta_{i,j}}\sigma_j$
for $i=1,\dots, n$.
We set $X_i = x_i$ and $Y_i = y_i$ for each $i$ and assume that the
relations (W1)-(W4) hold  in $\mathsf{A} = \mathsf{D}(\un \phi, \un t)$ for these choices.    
Let $\Phi$ denote
the group generated by the automorphisms $\phi_i$. 

Let $\mathfrak{max}\Df$ denote the set of  maximal ideals of $\Df$.   A module
for $\Df(\un \phi, \un t)$ is said to be a \emph {weight module}  if  $\mathsf{V} = \bigoplus_{\mathfrak m} \mathsf{V}_{\mathfrak m}$
where $\mathsf{V}_{\mathfrak m} = \{ v \in \mathsf{V} \mid \mathfrak m v = 0\}$.     The maximal ideal
$\mathfrak m$ is a \emph{weight} of $\mathsf{V}$ if $\mathsf{V}_{\mathfrak m} \neq 0$.  
The support $\mathfrak{supp}(\mathsf{V})$  is the set of weights of $\mathsf{V}$.    It is 
easy to verify that $X_i.\mathsf{V}_{\mathfrak m} \subseteq \mathsf{V}_{\phi_i(\mathfrak m)}$
and $Y_i .\mathsf{V}_{\mathfrak m} \subseteq \mathsf{V}_{\phi_i^{-1}(\mathfrak m)}$.
Thus, each weight module $\mathsf{V}$ can be decomposed into a direct sum of $\mathsf{A}$-submodules:

$$\mathsf{V}= \bigoplus_{\mathcal O}  \mathsf{V}_{\mathcal O},   \quad   \quad 
\mathsf{V}_{\mathcal O} = \bigoplus_{\mathfrak m \in \mathcal O} \mathsf{V}_{\mathfrak m},$$

\noindent   where $\mathcal O$ runs over the $\Phi$-orbits of $\mathfrak{max}\Df$.   
Consequently,  the category $\mathcal W(\mathsf{A})$ of weight modules for $\mathsf{A}$ decomposes
into a direct sum of full subcategories corresponding to the orbits of $\Phi$.  In particular,
an indecomposable weight module must have weights belonging to a single orbit $\mathcal O$.
Let  $\mathcal W_{\mathcal O}(\mathsf A)$  denote the subcategory of weight modules for $\mathsf{A}$
whose support lies in the orbit $\mathcal O$.      Our aim is to show that $\mathcal W_{\mathcal O}(\mathsf{A})$
is equivalent to a certain category $\mathcal C_{\mathcal O}$.   We then  use results
developed in \cite{BBF}  and \cite{DGO} to give an explicit realization of the
simple modules in $\mathcal W_{\mathcal O}(\mathsf{A})$.  \medskip

  {\it  Henceforth we will assume $\mathbb K$ is an algebraically closed field.}
 \m
 
Our  assumption on the field implies that each maximal ideal $\mathfrak m$ of $\Df$ has the form 
$\mathfrak m = \langle \rho_i - \mu_i, \sigma_i - \nu_i \mid i = 1,\dots,n\rangle$ for nonzero
scalars $\mu_i, \nu_i \in \mathbb K$.   
\m

\begin{subsec}{\bf The category $\mathcal C_{\mathcal O}$} \label{subsec:co}  \end{subsec}

Observe that if $\tau \in \Phi$ and $\tau(\mathfrak m) = \mathfrak m$ for  $\mathfrak m
= \langle \rho_i - \mu_i, \sigma_i - \nu_i \mid i = 1,\dots,n\rangle \in \mathfrak{max}\Df$, 
then $\tau = 1$.    Indeed, if $\tau = \phi_1^{k_1} \cdots \phi_n^{k_n}$,   then
$\tau(\rho_i-\mu_i) = r_i^{-k_i} \rho_i - \mu_i \in \mathfrak m$, hence $\rho_i - r_i^{k_i}\mu_i \in \mathfrak m$, implying that $r_i^{k_i} \mu_i = \mu_i$.   But since $\mu_i \neq 0$ and $r_i$ is not 
a root of unity, it must be that $k_i = 0$ for each $i$.   Thus, $\tau = 1$.   This says, in the
language of \cite{BvO} and \cite{BBF},  that every orbit $\mathcal O$ is {\it linear}.

\smallskip   When $t_j$ belongs to the maximal ideal $\mathfrak m = \langle \rho_i - \mu_i, \sigma_i - \nu_i \mid i = 1,\dots,n\rangle$, we say that $\mathfrak m$ has a {\it break} at $j$.        Now 
\begin{eqnarray}  t_j =  \displaystyle{\frac{r_j^2 \rho_j^2 - s_j^2 \sigma_j^2}{r_j^2 - s_j^2}} \in
\mathfrak m & \Longleftrightarrow &   r_j^2 \rho_j^2 - s_j^2 \sigma_j^ 2 \in \mathfrak m \\
& \Longleftrightarrow &  r_j^2 \mu_j^2 - s_j^2 \nu_j^2 = 0 \nonumber \\
& \Longleftrightarrow &  \nu_j = \pm r_j s_j^{-1}\mu_j. \nonumber  \end{eqnarray}
 
Assume that the maximal ideal $\mathfrak m = \langle \rho_i - \mu_i, \sigma_i - \nu_i \mid i = 1,\dots,n\rangle$  has the largest number of breaks among all the maximal ideals in the orbit
$\mathcal O$.       Suppose that the breaks occur at the values in the set
$J = \{j_1,\dots, j_q\}$.  (Possibly $J$ is empty.)            Thus, from the above
computation, we know that $j \in J$ if and only if  $\nu_j = \pm r_j s_j^{-1} \mu_j$.   
Now assume that $k \in \{1,\dots, n\} \setminus J$ and that some 

$$\phi_1^{p_1} \dots \phi_n^{p_n} (\mathfrak m) = \langle \rho_i - r_i^{p_i} \mu_i, \sigma_i - s_i^{p_i}\nu_i \mid i=1\dots n\rangle$$ has a break at $k$.     Then
$s_k^{p_k} \nu_k =  \pm r_k s_k^{-1} r_k^{p_k} \mu_k$, or 
$\nu_k = \pm (r_k s_k^{-1})^{p_k+1} \mu_k$.   But then  the maximal ideal 

$$\phi_k^{p_k}(\mathfrak m)    
= \langle \rho_i - \mu_i, \sigma_i - \nu_i  \ (i \neq k),  \
\rho_k-r_k^{p_k} \mu_k, \sigma_k - s_k^{p_k} \nu_k   \rangle$$

\noindent has more breaks than $\mathfrak m$.      This contradiction shows
that for the maximal ideals in the orbit $\mathcal O$, there are no breaks
outside of the set $J$.    Thus, in the terminology of \cite{BBF}, the maximal ideal $\mathfrak m$
is a {\it maximal break with respect to the set $J$}.    The set $J$ contains the set of 
breaks of every element in $\mathcal O$.  

Now with $\mathfrak m =   \langle \rho_i -  \mu_i, \sigma_i -  \nu_i \mid i=1\dots n\rangle$
 in $\mathcal O$ having the largest  number of breaks,  observe that
 for any $\mathfrak  n \in \mathcal O$, there exists an
 automorphism in $\Phi$ sending $\mathfrak m$ to $\mathfrak n$,
 and that automorphism is unique by our assumptions on the parameters  $r_i$ and $s_i$,
 so we denote it  $\phi_{\mathfrak n}$.   Note also  that
 $\phi_i \phi_{\mathfrak n} = \phi_{\phi_i(\mathfrak n)}$.
 
 For each $\phi_{\mathfrak n}$, we have an 
induced isomorphism $\Df/\mathfrak m \rightarrow \Df/\mathfrak n$, which we
again denote 
$\phi_{{\mathfrak n}}$, given by $\phi_{{\mathfrak n}}(d + \mathfrak m) = \phi_{\mathfrak n}(d) 
+ \mathfrak n$.  Let $\phi_{\mathfrak n}^{-1}: \Df/\mathfrak n \rightarrow \Df/\mathfrak m$ be
the 
inverse isomorphism.  

\begin{defn}\label{defn:co}   The category  ${\mathcal C}_{\mathcal  O}$ is the
$\K$-category whose objects are the maximal ideals in 
$ {\mathcal O}$ generated over $\K$ by
the set of 
morphisms $\{X_{{\mathfrak n},i}, Y_{{\mathfrak n},i}\, |\,{\mathfrak n}\in {\mathcal  O},
1\leq i\leq 
n\}$, where $X_{{\mathfrak n}, i}: {\mathfrak n}\to \phi_i({\mathfrak n})$ and
$Y_{{\mathfrak n}, i}: 
\phi_i({\mathfrak n})\to {\mathfrak n}$, subject to the relations:
 \begin{eqnarray}\label{eq:corels}   && X_{{\mathfrak n},i}\lambda =\lambda X_{{\mathfrak n},i},\
Y_{{\mathfrak n},i}\lambda =\lambda Y_{{\mathfrak n},i}, \\
&& Y_{{\mathfrak n},
i}X_{{\mathfrak n}, i}= 
\phi_{{\mathfrak n}}^{-1}(\overline t_i) 1_{{\mathfrak n}},   \quad \hbox{\rm where}
\ \ \overline t_i = t_i + \mathfrak n,   \nonumber   \\    
&&X_{{\mathfrak n}, i}Y_{{\mathfrak n}, 
i}= \phi_{{\mathfrak n}}^{-1}(\overline t_i)
1_{\phi_i(\mathfrak n)},  \nonumber  \end{eqnarray}  for each $\lambda \in \K$, ${\mathfrak n}\in {\mathcal O}$, and
$\ 1\leq 
i\leq n$.    \end{defn}

Let  ${\mathcal C}_{\mathcal O}$-$\Mod$ be the category of $\K$-linear additive functors
$\mathsf{M}: {\mathcal C}_{\mathcal O}  \rightarrow \K$-$\Mod$ into the category  
of $\K$-vector spaces.   Thus, $\mathsf{M}(\mathfrak n)$ is a $\K$-vector space for each
$\mathfrak n \in \mathfrak{max}\Df$, and $\mathsf{M}(a)u \in \mathsf{M}(\mathfrak p)$ for all  morphisms $a \in
{\mathcal C}_{\mathcal O}(\mathfrak n, \mathfrak p)$ and all $u \in \mathsf{M}(\mathfrak n)$.  
To make the action appear more module-like, we write $au$ rather than  $\mathsf{M}(a)u$.

\begin{prop} \label{prop1} Let $\mathsf{A} = \Ars = \Df(\un \phi,\un t)$,  and let ${\mathcal O}$ be an orbit of $\mathfrak{max} \Df$.   Then 
${\mathcal W}_{{\mathcal O}}(\mathsf{A})\cong {\mathcal C}_{\mathcal O}$-$\Mod$.   
\end{prop}
 
\begin{proof}  The proof is similar to that of Proposition 2.2 in
\cite{DGO}.   Compare also the proof of Proposition 3.4 of \cite{BBF}.   
We assume $\mathfrak m$ is the designated maximal ideal of
$\mathcal O$ having the largest number of breaks.  
We implicitly identify $\K$ with $\Df/\mathfrak m$ in the following.  
 Let $\mathsf{V} = \bigoplus_{\mathfrak n \in {\mathcal O}}\mathsf{V}_{\mathfrak n}$ belong
to 
${\mathcal W}_{{\mathcal O}}(\mathsf{A})$.  For each $\mathfrak n \in {\mathcal O}$, set
$\mathsf{M}_\mathsf{V}(\mathfrak n) = \mathsf{V}_{\mathfrak n}$.  Using the 
isomorphism $\phi_{\mathfrak n}: \Df/{\mathfrak m} \rightarrow \Df/{\mathfrak n}$, 
we can view $\mathsf{M}_\mathsf{V}(\mathfrak n)$ 
as a $(\Df/\mathfrak m)$-vector space via $\overline d v := \phi_{\mathfrak n}(\overline d)v$, 
($\overline  d = d + \mathfrak m$).  For $v\in \mathsf{M}_\mathsf{V}(\mathfrak n)$ 
and $w\in \mathsf{M}_\mathsf{V}(\phi_i(\mathfrak n))$, 
we define $X_{\mathfrak n,i}v: = X_iv\in \mathsf{M}_\mathsf{V}(\phi_i(\mathfrak n))$ 
and $Y_{\mathfrak n,i}w: = 
Y_iw \in \mathsf{M}_\mathsf{V}(\mathfrak n)$. Then for  $\overline d \in \Df/\mathfrak m$, we have 
$X_{\mathfrak n,i}\overline d v = X_i\phi_\mathfrak n(\overline d)v = 
\phi_i(\phi_{\mathfrak n}(\overline d))X_iv = \overline d X_{\mathfrak n,i}v$, 
and $Y_{\mathfrak n, i}\overline d w= \overline d Y_{\mathfrak n,i}w.$ We also have 
$Y_{\mathfrak n,i}X_{\mathfrak n,i}v=Y_iX_iv = t_iv = (t_i+\mathfrak n)v 
= \phi_\mathfrak n^{-1} (\overline  t_i)v$ for $\overline  t_i = t_i + \mathfrak n$, 
and $X_{\mathfrak n,i}Y_{\mathfrak n,i}w= X_i Y_i w 
= \phi_i(t_i)w = \big(\phi_i(t_i)+ \phi_i(\mathfrak n)\big)w = 
\phi^{-1}_{\phi_i(\mathfrak n)}(\phi_i(\overline t_i))w= \phi_{{\mathfrak 
n}}^{-1}(\overline t_i)w$.
Hence, $\mathsf{M}_\mathsf{V}$ ($\mathsf{M}_\mathsf{V}: {\mathcal C}_{\mathcal O} \rightarrow 
(\Df/\mathfrak m)$-${\Mod}$) is a ${\mathcal C}_{\mathcal O}$-module, 
and we have the functor

\begin{equation} \mathsf{F} : {\mathcal W}_{{\mathcal O}}(\mathsf{A})\to
{\mathcal C}_{\mathcal O}\mbox{\text-}\Mod, \qquad \mathsf{V} 
\mapsto \mathsf{M}_\mathsf{V}.  \end{equation} 

Conversely, for each $\mathsf{M}\in {\mathcal C}_\mathcal O$-$\Mod$, 
let $\mathsf{V}_\mathsf{M}:=\bigoplus_{\mathfrak n\in
{\mathcal O}}\mathsf{M}(\mathfrak n)$ with the action of $\mathsf{A}$ on $\mathsf{V}_\mathsf{M}$ specified by 
$dv:=\phi_{\mathfrak n}^{-1}(\overline{d})v$,\ $\overline d = d + \mathfrak n \in
\Df/\mathfrak n$, 
$X_iv:=X_{\mathfrak n,i}v$, and
$Y_iv:=Y_{\phi_i^{-1}(\mathfrak n),i}v$ for 
$v\in \mathsf{M}(\mathfrak n)$.   Then $\mathsf{V}_\mathsf{M}= \bigoplus_{\mathfrak n \in \mathcal O}  (\mathsf{V}_\mathsf{M})_{\mathfrak n}
\in {\mathcal W}_{{\mathcal O}}(\mathsf{A})$,  where 
$ (\mathsf{V}_\mathsf{M})_{\mathfrak n} = \mathsf{M}(\mathfrak n)$.   
Thus,  
\begin{equation} \mathsf{F}^{\prime}:{\mathcal C}_{\mathcal O}\mbox{\text-}\Mod
\rightarrow 
{\mathcal W}_{{\mathcal O}}(\mathsf{A}), \qquad \mathsf{M} \mapsto \mathsf{V}_\mathsf{M},
\label{fprime}\end{equation} is a functor which is 
inverse to $\mathsf{F}$.    \end{proof} 
\vspace{.05truein}

 \begin{pgraph} {\rm We introduce an
equivalence relation $\sim$ on the set of 
maximal ideals  in $\mathfrak {max} D$.  This relation is the
transitive 
extension of the relation specified by the following
conditions: \ $\mathfrak n 
\sim \mathfrak n$; \ $\mathfrak n \sim \mathfrak p$ implies $\mathfrak p \sim \mathfrak n$;  and  $\mathfrak n \sim \phi_i(\mathfrak n)$ if and 
only if $t_i \not \in \mathfrak n$.} \end{pgraph}

\begin{lem}\label{lem:iso}  Assume $\mathfrak n$ and $\mathfrak p$ belong to $\mathcal O$.  Then $\mathfrak n
\sim \mathfrak p$ if and only if $\mathfrak n $ and $\mathfrak p$ are isomorphic in $\mathcal C_{\mathcal O}$.    \end{lem}

\begin{proof}  We assume $\mathfrak m$ is the designated maximal ideal of
$\mathcal O$ having the largest number of breaks.    For $\mathfrak n \in \mathcal O$, let
$\phi_{\mathfrak n}$ denote the unique element of the group $\Phi$ such that
$\phi_{\mathfrak n}(\mathfrak m) = \mathfrak n$.       \m

 \noindent ($\Longrightarrow$) \  It suffices to consider the case  that  $\mathfrak n \sim 
\phi_i(\mathfrak n)$.   Then  $t_i \not \in \mathfrak n$  
so that  $t_i + \mathfrak n \neq 0$, and 
since $\phi_{\mathfrak n}^{-1}: \Df/\mathfrak n \rightarrow \Df/\mathfrak m$ is an isomorphism, 
$\phi_{\mathfrak n}^{-1}(t_i + \mathfrak n) \neq 0$.   
It follows from  \eqref{eq:corels}  that  
$X_{\mathfrak n,i}$ and $Y_{\mathfrak n,i}$ are invertible in ${\mathcal C}_\mathcal O$.  
Hence $\mathfrak n$ and $\phi_i(\mathfrak n)$ are isomorphic in ${\mathcal C}_\mathcal O$.
\smallskip

\noindent ($\Longleftarrow$)   Assume now
that $\mathfrak n$ and $\phi_i(\mathfrak n)$
are isomorphic in ${\mathcal C}_{\mathcal O}$.
Then it 
follows {f}rom the definition of the morphisms 
in ${\mathcal C}_{\mathcal O}$ that 
\begin{eqnarray*} &&{\mathcal C}_{\mathcal O}(\mathfrak n,\phi_i(\mathfrak n)) = {\mathcal 
C}_{\mathcal O}(\phi_i(\mathfrak n),\phi_i(\mathfrak n))X_{\mathfrak n,i}=
X_{\mathfrak n,i}{\mathcal C}_{\mathcal O}(\mathfrak n,\mathfrak n)  \qquad \hbox{\rm and} \\
&&{\mathcal C}_{\mathcal O}(\phi_i(\mathfrak n),\mathfrak n) = 
{\mathcal C}_{\mathcal O}(\mathfrak n,\mathfrak n)Y_{\mathfrak n,i}=
Y_{\mathfrak n,i}{\mathcal C}_{\mathcal O}(\phi_i(\mathfrak n),\phi_i(\mathfrak n)) \end{eqnarray*}
\noindent for $1\leq 
i\leq n$.  
Hence $X_{\mathfrak n,i}$ and $Y_{\mathfrak n,i}$ are isomorphisms.  
If $t_i \in \mathfrak n$, then $\phi_{\mathfrak n}^{-1}(t_i + \mathfrak n) = 0
\in \Df/\mathfrak m$ and 
$Y_{\mathfrak n,i}X_{\mathfrak n,i} = \phi_{\mathfrak n}^{-1}(t_i+\mathfrak n) 1_{\mathfrak n} =
0$, a 
contradiction.  Thus, $t_i \not \in \mathfrak n$ and
$\mathfrak n \sim \phi_i(\mathfrak n)$.   \end{proof}
\m
 
\begin{subsec}{\bf A skeleton of \ $\mathcal C_{\mathcal O}$}  \end{subsec}

A category ${\mathcal C}$ is said to be {\em basic} if 
\begin{itemize}
\item[{$\bullet$}]  all its objects are pairwise nonisomorphic;

\item[{$\bullet$}]  for each object $\alpha$, there are no
nontrivial idempotents
in ${\mathcal C}(\alpha,\alpha)$.
\end{itemize}  

 A full subcategory ${\mathcal S}$ is a {\em skeleton}
of a category
${\mathcal C}$ if it is basic, and each object $\alpha
\in 
\Ob{\mathcal C}$ is isomorphic to a direct summand of
a (finite) direct sum of some
objects of ${\mathcal S}$.    The natural inclusion
functor 
${\mathcal I}: {\mathcal S} \rightarrow {\mathcal C}$
of a skeleton 
${\mathcal S}$ into ${\mathcal C}$ is an equivalence
of categories.  By this
functor, ${\mathcal C}$ becomes a 
${\mathcal C}-{\mathcal S}$-bimodule in an obvious
way.  Tensoring this bimodule over
${\mathcal S}$ furnishes equivalences ${\mathcal
S}$-$\Mod \rightarrow
{\mathcal C}$-$\Mod$.  This is a reformulation of
Morita equivalence in the
categorical context.    \smallskip

Next we identify a skeleton for each category $\mathcal C_{\mathcal O}$.  
As above, let ${\mathfrak m}$ be the designated maximal ideal with the largest
number of breaks  in the orbit $\mathcal O$.      If the set $J$ of breaks for $\mathfrak m$ is
empty,    set $\mathcal B_{\mathcal O} = \{ \mathfrak m\}$.   If  $J = \{j_1,\dots, j_q\}$
is nonempty, then let  $\mathcal B_{\mathcal O} = \{ \phi^\alpha(\mathfrak m):= \phi_{j_1}^{\alpha_1} \cdots \phi_{j_q}^{\alpha_q} (\mathfrak m)
\mid \alpha_j \in \{0,1\} \ \hbox{\rm for all} \ j \in J\}$.    For each maximal ideal $\phi^\alpha (\mathfrak m)$ in $\mathcal B_{\mathcal O}$,   define 

\begin{equation}\label{eq:opdef}  \mathcal O_{\alpha} = \left \{ \phi_1^{k_1} \cdots \phi_n^{k_n}\phi^\alpha(\mathfrak m) \ \Bigg | \
 k_j \in  \begin{cases}  \mathbb Z &  \ \hbox{\rm if} \  j \in J^c,  \\
\mathbb Z_{\leq 0}& \ \hbox{\rm if  \ 
$j \in J$ and  $\alpha_j = 0$},  \\
\mathbb Z_{\geq 0} &  \ \hbox{\rm if  \ 
$j \in J$ and  $\alpha_j = 1$}.  \end{cases} \right \}\end{equation}
\m

\begin{rem} {\rm For $\phi^\alpha(\mathfrak m) \in \mathcal{B}_{\mathcal O}$, the
equivalence class of
$\phi^\alpha(\mathfrak m)$ is exactly ${\mathcal O}_{\alpha}$, and $\mathcal O$ is the disjoint
union of the sets
$\mathcal O_{\alpha}$,  $\alpha \in \{0,1\}^{|J|}$.} \end{rem}
\m

 \begin{prop} \label{cor:one} The full
subcategory ${\mathcal 
S}_{\mathcal O}$ of $\mathcal C_{\mathcal O}$  with $\Ob{\mathcal S}_{\mathcal O} = \mathcal{B}_{\mathcal O}$
is a skeleton of 
the category $\mathcal C_{\mathcal O}$.  \end{prop}

\begin{proof}   Assume $\mathfrak m$ is the designated maximal ideal in
$\mathcal O$, and the 
corresponding set of breaks is $J =
\{j_1, \dots, j_q\}$  (which  may be empty).
If $J \neq \emptyset$,  then 
$\phi^\alpha(\mathfrak m) =\phi_{j_1}^{\alpha_1}\cdots\phi_{j_q}^{\alpha_q}(\mathfrak m)$
has
breaks at $\{j_k \mid \alpha_k = 0\}$.  Thus, these
objects
are pairwise nonisomorphic.  Moreover, if $\mathfrak n =
\phi_1^{\ell_1} \cdots 
\phi_n^{\ell_n}(\mathfrak m) \in \mathcal O$,   then by Lemma \ref{lem:iso}, $\mathfrak n$
is isomorphic to 
$\phi_{j_1}^{\alpha_1}\cdots\phi_{j_q}^{\alpha_q}(\mathfrak m)
\in 
{\mathcal B}_{\mathcal O}$ such that $\alpha_j = 1$ whenever $j
\in J$ and $\ell_j \geq 1$,   and $\alpha_j = 0$ otherwise.  If  $J = \emptyset$  and 
${\mathcal B}_{\mathcal O} = \{\mathfrak m\}$,  then 
every element of 
$\mathcal O$ is isomorphic to $\mathfrak m$ by the lemma. 
Consequently, in both cases 
${\mathcal B}_{\mathcal O}$ is a skeleton of $\mathcal C_{\mathcal O}$.  \end{proof} 
\smallskip

\begin{subsec} {\bf A quiver description of the
skeleton} \end{subsec}  
 
For a  field $\K$ and an arbitrary subset $J$ 
of positive integers, we define the
category 
${\mathfrak A}={\mathfrak A}(\K, J)$ as the $\K$-linear category with
the set of objects 
$\Ob{\mathfrak A}:=\{0,1\}^{|J|}$ generated (over $\K$) by the set of morphisms 
${\mathfrak A}_1:=\{a_{\alpha,j},b_{\alpha,j}\mid {\alpha}\in \Ob{\mathfrak A},\
j \in J \}$, where $a_{{\alpha}, j}: {\alpha}\to \beta$ and
$b_{{\alpha}, j}: 
\beta\to \alpha$  are such that $\beta_k=\alpha_k$ for all $k\ne j$,  $\alpha_j = 0$
and $\be_j=1$, 
subject to the relations:

\begin{itemize}
\item $a_{{\alpha}, j}b_{{\alpha}, j}=b_{{\alpha}, j}a_{{\alpha},
j}=0$ \ 
for each $a_{\alpha,j}, b_{\alpha,j} \in {\mathfrak A}_1$; 
\item $u_{{\alpha}, j}v_{{\be}, k}-v_{{\gamma},
k}u_{{\delta}, j}=0$ \
for all $k\neq j$
and all possible $u,v\in \{a, b\}$,
${\alpha,\beta,\gamma,\delta}\in \Ob{\mathfrak A}$, 
for which
the last equality makes sense.
\end{itemize}
 
\smallskip

\noindent  When $J$ is empty, let
${\mathfrak A}\,(\K,\emptyset)$ be the category with 
a unique object, say $\omega$, and with morphism set
$\K 1_\omega$.

\smallskip

The $\K$-algebra corresponding to the category
${\mathfrak A}(\K, 
J)$ above consists of finite $\K$-linear combinations
of morphisms in the 
category, and the product is simply composition of
morphisms whenever 
it is defined and is 0 otherwise.  It has a unit
element.  We adopt the same notation ${{\mathfrak
A}}(\K,J)$ for the algebra, 
as it will be evident {f}rom the context which one is
meant.  When $J$ is nonempty, say $J=\{j_1, \cdots, j_q\}$, then ${\mathfrak
A}(\K,J)\cong {\mathfrak A}(\K, 
\{j_1\})\otimes \cdots \otimes {\mathfrak
A}(\K,\{j_q\})$.  It is easy to 
see that algebra ${\mathfrak A}(\K,\{j\})$ is isomorphic to the
algebra 
$\overline {\mathfrak Q}_1:= \K {\mathfrak
Q}_1/{\mathfrak R}$ corresponding to the 
following quiver and relations:

\begin{picture}(0.00,40.00)
\put(00.00,17.00){${\mathfrak Q}_1:$}
\put(87.00,27.00){$a$} 
\put(87.00,6.50){$b$} \put(69.00,27.00){$1$}
\put(101.00,7.00){$2$} 
\put(180.00,17.00){$ab=ba=0.$}
\put(69.50,17.00){$\circ$} 
\put(101.00,17.00){$\circ$}
\put(74.00,22.00){\vector(1,0){27.00}} 
\put(101.50,16.50){\vector(-1,0){27.00}}
\end{picture}

\noindent As $\overline {\mathfrak Q}_1$ is generated
over $\K$ by $1_1, 1_2, a, 
b$, modulo the relations $ab 
= ba = 0$, it has dimension 4. (Here and throughout
the paper we do not
list obvious relations such as $a^2 = 0$, $1_1^2 =
1_1$, 
$a 1_1 = a$, etc.)     \smallskip

\begin{prop} \label {skelrel} Let $\mathsf{A}  = \Ars = \Df(\un \phi,\un t)$, and   
assume $\mathcal O$ is an orbit in $\mathfrak{max}\Df$
under the 
automorphism group $\Phi$ generated by the $\phi_i$.
Let $\mathfrak m$ be the designated maximal ideal of $\mathcal O$ having the
largest set of breaks.   Then
${\mathcal S}_{\mathcal O} \cong {\mathfrak A}(\Df/\mathfrak m, J)$,
where $\mathcal S_{\mathcal O}$ is the full 
subcategory of  $\mathcal C_{\mathcal O}$ in
Proposition \ref{cor:one}  and $J$ is the set of breaks of $\mathfrak m$.  
  \end{prop}

\begin{proof}  Proposition \ref{cor:one} shows that the category
${\mathcal S}_{\mathcal O}$ with 
objects $\mathcal B_{\mathcal O}$ is a skeleton of ${\mathcal
C}_{\mathcal O}$.    When $J$ is nonempty, say $J = \{j_1,\dots,j_q\}$, 
define the functor  \  $\mathsf{G}: {\mathfrak A}(\Df/\mathfrak m, J) 
\rightarrow {\mathcal S}_{\mathcal O}$  \  as follows:

\begin{equation}\mathsf{G}(\alpha) =  \phi^\alpha(\mathfrak m) \qquad \mathsf{G}(a_{\alpha,j}) = 
X_{\phi^{\alpha}(\mathfrak m),j}, \qquad  \mathsf{G}(b_{\alpha,j}) = 
Y_{\phi^{\alpha}(\mathfrak m),j},  \label{g1}  \end{equation}
\noindent where $ \phi^\alpha =\prod_{j_i \in J} \phi_{j_i}^{\alpha_i}$.
 {F}rom 
subsection (\ref{subsec:co}) it is easy to see that this
is an isomorphism.   Now suppose that $J = \emptyset$.  Then  the functor $\mathsf{G}: 
{\mathfrak A}(\K,\emptyset) \rightarrow {\mathcal S}_{\mathcal O}$ is 
defined by
\begin{equation} \mathsf{G}(\omega) = \mathfrak m, \qquad \mathsf{G}(1_\omega) =
1_{\mathfrak m}.   \label{g2}
\end{equation} \end{proof}

Proposition \ref{skelrel} allows us to focus on
the algebras 
of the form ${\mathfrak A}(\K, J)$, where $J$ is a subset
of $\{1,\dots, n\}$.     For each
$\alpha \in 
\Ob{\mathfrak A}(\K, J)$, define a simple ${\mathfrak A}(\K, J)$-module
${\mathfrak S}_\alpha$ such 
that ${\mathfrak S}_\alpha(\beta)=\delta_{\alpha,\beta}\K$ for all
objects $\beta 
\in \Ob{\mathfrak A}(\K, J)$, and let all morphisms be
trivial.   Then the following result is clear.

\smallskip

\begin{prop}\label{aprop} 
Any simple module over the algebra ${\mathfrak A} ={\mathfrak A}(\K,J)$
is isomorphic to 
$\mathfrak S_{\alpha}$ for some object $\alpha\in \Ob{\mathfrak A}(\K,
J)$. \end{prop}  \smallskip

\begin{subsec}  {\bf $\Ars$-simple weight modules with
no breaks}  \end{subsec}

We apply the results of  the previous subsections to determine the simple modules
in $\mathcal W_{\mathcal O}(\mathsf{A})$  ($\mathsf{A} = \Ars = \Df(\un \phi,\un t)$) (compare 
\cite[\S 4]{BBF}). 
First  we assume that the maximal ideals of $\mathcal O$  have no breaks.  
Recall in this case, $\mathcal B_{\mathcal O} = \{\mathfrak m\}$, where $\mathfrak m$
is the designated maximal ideal (which could be any ideal in $\mathcal O$).

Set  
\begin{equation} \mathsf{Z}(\mathcal O) =\bigoplus_{\mathfrak n\in {\mathcal O}}\Df/\mathfrak n
\end{equation}

\noindent and define a left $\mathsf{A}$-module structure on
$\mathsf{Z}({\mathcal O})$ by specifying 
for $i=1,\dots,n$ and $d' \in \Df$ that 
\begin{eqnarray}\label{eq:actmax}  d' (d + \mathfrak n) &=& d'd + \mathfrak n,  \\
X_i (d+\mathfrak n) &=& \phi_i(d)+\phi_i(\mathfrak n),  \nonumber \\
Y_i (d+\mathfrak n) &=&  t_i\phi^{-1}_i(d)+\phi^{-1}_i(\mathfrak n).  \nonumber 
\end{eqnarray}

\noindent As $\mathsf{Z}(\mathcal O)$ is generated by $1 + \mathfrak m$, we have
that $\mathsf{Z}(\mathcal O) 
\cong \mathsf{A}/\mathsf{A}\mathfrak m$ where $1+\mathfrak m \mapsto 1+\mathsf{A}\mathfrak m$.

We want a more explicit realization of this module, and for this we suppose that
$\mathfrak m = \langle \rho_i - \mu_i,  \sigma_i - \nu_i \mid i=1,\dots, n\rangle$ where
$\mu_i,\nu_i$ are nonzero scalars in $\K$ and $\nu_i \neq \pm r_is_i^{-1}\mu_i$
for any $i$ (no breaks for $\mathfrak m$).     If $\mathfrak n \in \mathcal O$, 
then $\mathfrak n = \phi_{\mathfrak n}(\mathfrak m)$, where  
$\phi_{\mathfrak n} = \phi_1^{k_1} \cdots \phi_n^{k_n}$ for some $k_i \in \mathbb Z$. 
Thus, $\mathfrak n = \langle \rho_i - r_i^{k_i}\mu_i,  \sigma_i - s_i^{k_i}\nu_i \mid i=1,\dots, n\rangle$ 
 Let $\un k = (k_1,\dots, k_n)
\in \mathbb Z^n$.   Here we will write  $\mathsf{Z}(\mu, \nu)$
for the module $\mathsf{Z}(\mathcal O)$, where $\mu = (\mu_1, \dots, \mu_n)$, $\nu
= (\nu_1,\dots, \nu_n) \in (\K^\times)^n$, and we will 
 identify the one-dimensional $\K$-subspace  $\Df/\mathfrak n$ with 
$\K z(\un k)$.       Then the relations  in \eqref{eq:actmax} translate to give
\begin{eqnarray}\label{eq:actz}  \rho_i z(\un k) &=& r_i^{k_i}\mu_i z(\un k), \\
\sigma_i z(\un k) &=& s_i^{k_i} \nu_i z(\un k), \nonumber \\
x_i z(\un k)&=& z(\un k + \epsilon_i),  \nonumber \\
y_i z(\un k) &=&\left( \frac{r_i^{2k_i}\mu_i^2 - s_i^{2k_i}\nu_i^2}{r_i^2 -s_i^2}
\right )  z(\un k - \epsilon_i).  \nonumber 
\end{eqnarray}

\begin{rem}  The polynomial module $\mathsf{P}(n)$ with $\Ars$-action given
by \eqref{eq:acts} is just the simple module $\mathsf{Z}(\un 1, \un 1)$ labeled by
the $n$-tuples of all ones.  \end{rem} \smallskip
 
\begin{thm}\label{thm:nobreaks}  Assume $\K$ is an algebraically closed field,  and let $\mathsf{A} = \Ars$ where the
parameters $r_i,s_i$ satisfy Assumption \ref{asspt1}.    Let  $\mu = (\mu_1, \dots, \mu_n)$, $\nu = (\nu_1, \dots, \nu_n) \in (\K^\times)^n$,
where $\nu_i \neq \pm r_i s_i^{-1} \mu_i$ for any $i$.   Then   
$\mathsf{Z}(\mu, \nu) = \bigoplus_{\un k \in \mathbb Z^n}  \K z(\un k)$ with $\mathsf{A}$-action
given by \eqref{eq:actz} is a  simple weight module for $\mathsf{A}$ with weights in 
the $\Phi$-orbit $\mathcal O$ determined by the maximal ideal $\mathfrak m = 
 \langle \rho_i - \mu_i,  \sigma_i - \nu_i \mid i=1,\dots, n\rangle$.     The module $\mathsf{Z}(\mu,\nu)$
 is the unique (up to isomorphism) simple $\mathsf{A}$-module in the category $\mathcal W_{\mathcal O}(\mathsf{A})$.
 \end{thm}   \smallskip

\begin{subsec}  {\bf $\Ars$-simple weight modules with
breaks}  \end{subsec}

Here we assume $\mathcal O$ is an orbit of $\mathfrak {max}\Df$,  and  the maximal ideal $\mathfrak m$ is chosen from $\mathcal O$  to have 
the largest number of breaks.   We suppose the breaks occur at the values  $J = \{j_1, \dots, j_q\}$
and $J \neq \emptyset$.   
Recall that $\mathcal B_{\mathcal O} = \{ \phi^\alpha(\mathfrak m):= \phi_{j_1}^{\alpha_1} \cdots \phi_{j_q}^{\alpha_q} (\mathfrak m)
\mid \alpha_j \in \{0,1\} \ \hbox{\rm for all} \ j \in J\}$ in this case.    For each maximal ideal $\phi^\alpha (\mathfrak m)$ in $\mathcal B_{\mathcal O}$,   let

\begin{equation} \label{eq:sop}  \mathsf{Z}({\mathcal O}, \alpha ):=\bigoplus_{\mathfrak n\in
{\mathcal O}_{\alpha}}\Df/\mathfrak n,   \end{equation}

\noindent where $\mathcal O_{\alpha}$ is as in \eqref{eq:opdef}.     One
can define the 
structure of a left $\mathsf{A}$-module on $\mathsf{Z}(\mathcal O,\alpha)$ by the
same formulae as 
in \eqref{eq:actmax}, but when the image is not in
$\mathsf{Z}(\mathcal O,\alpha)$, the result is 0. We
have in this case $\mathsf{Z}(\mathcal O,\alpha) 
\cong \mathsf{A}/\mathsf{A}(\phi^\alpha(\mathfrak m), Z_{j_1}, \ldots, Z_{j_q})$ where
$Z_{j_k}=X_{j_k}$ if $\alpha_k = 0$ (i.e. $\phi^\alpha(\mathfrak m)$ has a 
break with respect to $j_k$), and $Z_{j_k}=Y_{j_k}$
otherwise.  The 
isomorphism is given by $1+\phi^\alpha(\mathfrak m) \mapsto 1 + \mathsf{A}(\phi^\alpha(\mathfrak m),
Z_{j_1}, \ldots, 
Z_{j_q})$.  It follows {f}rom the construction that $\mathsf{Z}({{\mathcal O}, \alpha})$ is a  simple $\mathsf{A}$-module for each such $\alpha \in \{0,1\}^{|J|}$.

 \begin{thm}\label{thm:breaks}  Let $\K$ be an algebraically closed field,  and let $\mathsf{A} = \Ars$,  where the parameters $r_i,s_i$ satisfy Assumption \ref{asspt1}.    Let  $\mu = (\mu_1, \dots, \mu_n)$, $\nu = (\nu_1, \dots, \nu_n) \in (\K^\times)^n$,
where $\nu_j = \pm  r_j s_j^{-1} \mu_j$ if and only if  $j \in J = \{j_1,\dots, j_q\} \neq \emptyset$.  Let $\mathcal O$ be the $\Phi$-orbit determined by the maximal ideal $\mathfrak m = 
 \langle \rho_j- \mu_j,  \sigma_j - \nu_j \mid j=1,\dots, n\rangle$.
Suppose  $\alpha = (\alpha_{j_1},\dots,\alpha_{j_q})
\in \{0,1\}^{|J|}$, and let  $\mathcal J_\alpha$ be the set of $\un k = (k_1,\dots,k_n)   \in \mathbb Z^n$  satisfying the following properties:  if $j\in J$ and $\alpha_j = 1$, then $k_j \in  \mathbb Z_{\geq 0}$;
if $j \in J$ and $\alpha_j = 0$, then $k_j \in \mathbb Z_{\leq 0}$; and   if $j \in J^c$, then $k_j
\in \mathbb Z$.     Then   
$\mathsf{Z}_{J,\alpha}(\mu, \nu) = \bigoplus_{\un k \in \mathcal J_\alpha}  \K z(\un k)$ with $\mathsf{A}$-action
given by \eqref{eq:actz} is a  simple weight module for $\mathsf{A}$ with weights in 
$\mathcal O_\alpha$  (see \eqref{eq:opdef} ).   The modules $\mathsf{Z}_{J,\alpha}(\mu,\nu)$
for $\alpha \in \{0,1\}^{|J|}$
are the unique (up to isomorphism) simple $\mathsf{A}$-modules in the category $\mathcal W_{\mathcal O}(\mathsf{A})$.
 \end{thm}   
  
\begin{rem} In applying the relations in \eqref{eq:actz} to $\mathsf{Z}_{J,\alpha}(\mu,\nu)$,  it is to be understood that  $z(\un \ell) = 0$ whenever $\un \ell \not \in \mathcal J_\alpha$. 
The modules in Theorems \ref{thm:nobreaks} and \ref{thm:breaks} exhaust the
simple modules in  $\mathcal W_{\mathcal O}(\mathsf{A})$,  and they  become 
weight modules for $\mathsf{U}_{\un r, \un s}(\mathfrak{sl}_n)$ via the 
map in \eqref{eq:maps}.  
   \end{rem}
  \begin{rem}  The Verma modules $\mathsf{V}(\lambda, \zeta)$ of  Section 5 are just the special case when 
 $\mu_i = \lambda_i\zeta_i$ and $\nu_i = \lambda_i \zeta_i'$ for each $i$, where
 $\zeta_i = \zeta(\rho_i)$ and $\zeta_i' = \zeta(\sigma_i)$.   Thus $\nu_i = \pm \mu_i$
 for all $i$, and $J = \{1,\dots, n\}$.   If  $\alpha = (\alpha_1,\dots, \alpha_n) = \un 1$, then 
$\mathcal J_{\un 1} = \mathbb N^n$.    Using \eqref{eq:actz} and \eqref{eq:vermav}
 it is easy to see that $\mathsf{V}(\lambda,\zeta) \cong \mathsf{Z}_{J,\un 1}(\mu, \nu)$ 
for these choices.    \end{rem} 

\begin{section}{Whittaker modules for $\Ars$} \end{section} 

Kostant \cite{K} introduced a class of modules for finite-dimensional complex 
semisimple Lie algebras and called them Whittaker modules because of their connections with Whittaker equations in number theory.  In \cite{Bl}, Block showed that the simple modules for $\mathfrak{sl}_2$ over $\mathbb C$ are either  highest (or lowest) weight modules, Whittaker modules, or modules obtained by localization.   

For a generalized Weyl algebra $\mathsf{A} = \Df(\un \phi, \un t)$,  we say that $(\mathsf {W}, w)$ is a {\it Whittaker pair for $\mathsf{A}$ of type $\xi = (\xi_i), \, \xi_i \in \K^\times$},   if   $\mathsf{W}$ is an $\mathsf{A}$-module such that 
$\mathsf{W} = \mathsf{A}w$  and  $X_iw = \xi_i w$ for  all $i$.  
In \cite[Thm.~3.12]{BO},  isomorphism classes of Whittaker pairs of type $\xi$ for $\mathsf{A}$  were shown to be in bijection with
the $\phi$-stable ideals of $\Df$.    In this final section, we apply that result to describe
the Whittaker modules for $\mathsf{A} = \Ars$.  

We have seen that the multiparameter Weyl algebra $\mathsf{A} = \mathsf{A}_{\un r, \un s}(n)$ has a realization
as a generalized Weyl algebra $\mathsf{A} = \Df(\un \phi, \un t)$, and the proof of Theorem \ref{thm:Asimple}
shows that there are no nontrivial $\phi$-invariant ideals in $\Df$ under Assumption \ref{asspt1}.    
Thus, every Whittaker module for $\mathsf{A}$ of type $\xi$ must be simple.  In particular,  the universal
Whittaker module $\mathfrak{W}(\xi):= \mathsf{A} \otimes_{\K[X]} \K w_{ \xi}$ of type $\xi$ constructed in
\cite[Sec.~3]{BO}   must be simple. 
 Here $\K[X]$ is the subalgebra of $\mathsf{A}$ generated
by the $X_i, i=1,\dots,n$,  and $\K w_{\xi}$ is the $\K[X]$ module with
$X_i w_{\xi} = \xi_i w_{\xi}$ for all $i$.    Since every Whittaker module for $\mathsf{A}$ of type $\xi$
is a homomorphic image of $\mathfrak{W}(\xi)$, every  Whittaker module for $\mathsf{A}$
 of type $\xi$ is isomorphic to $\mathfrak{W}(\xi)$. By \cite[Rem.~3.6]{BO},  the vectors 
\begin{equation}\label{eq:Whitbasis} w(\un k, \un \ell) := \left(\prod_{i =1}^n \rho_i^{k_i} \prod_{j=1}^n \sigma_j^{\ell_j}\right) \otimes w_{\xi}\end{equation}
for $\un k, \un \ell \in \mathbb Z^n$
form  a basis  for $\mathfrak{W}(\xi)$.  The $\mathsf{A}$-action
is given by the following:
\begin{eqnarray}\label{eq:Whitact}  \rho_i w(\un k, \un \ell) &=& w(\un k + \epsilon_i, \un \ell), \\
\sigma_i w(\un k, \un \ell) &=&  w(\un k, \un \ell+\epsilon_i), \nonumber \\
X_i  w(\un k, \un \ell) &=& \xi_i\prod_{j=1}^n r_j^{-k_j} s_j^{-\ell_j}w(\un k, \un \ell), \nonumber  \\
Y_i w(\un k, \un \ell) &=& \nonumber  \\
&& \hspace{-.38truein} \xi_i^{-1} \prod_{j=1}^n r_j^{-k_j} s_j^{-\ell_j} 
\left(\frac{r_i^2}{r_i^2-s_i^2} w(\un k + 2 \epsilon_i, \un \ell) - \frac{s_i^2}
{r_i^2-s_i^2} w(\un k, \un \ell+2 \epsilon_i) \right). \nonumber    \end{eqnarray}

\begin{thm}\label {thm:Whit} When the parameters $r_i,s_i$ satisfy Assumption \ref{asspt1},  every Whittaker  module of type $\xi$ for the algebra $\mathsf{A} =\Ars$ is simple
and isomorphic to the universal Whittaker module   
$\mathfrak{W}({\xi}) = \mathsf{A} \otimes_{\K[X]} \K w_{\xi}$ with basis as in 
\eqref{eq:Whitbasis} and   $\mathsf{A}$-action given by \eqref{eq:Whitact},
where $x_i$ acts by $X_i$ and $y_i$ by $Y_i$ for all $i=1,\dots, n$.  \end{thm}
\m

\begin{center} {\textbf{Acknowledgments}}   \end{center}
This paper was written while the author was a Simons Visiting Professor in the Combinatorial Representation Theory Program at the Mathematical Sciences Research Institute.   The  hospitality of MSRI and the support from the Simons Foundation as well as from National Science Foundation grant \#{}DMS--0245082 are  gratefully acknowledged.

\end{document}